\documentclass[a4paper,12pt]{article}
\usepackage{amsmath,amsthm,amssymb,euscript,amscd}
\usepackage{graphicx}
\usepackage[lmargin=3.5cm,rmargin=3.5cm,tmargin=4cm,bmargin=4.0cm]{geometry}
\usepackage{xcolor}
\usepackage{enumerate}
\usepackage{tikz-cd}
\usepackage{enumitem}
\usepackage{nicefrac}
\usepackage{cleveref}
\newtheorem{theorem}{Theorem}[section]
\newtheorem*{theorem-non}{Theorem}

\newtheorem{corollary}[theorem]{Corollary}
\newtheorem{lemma}[theorem]{Lemma}
\newtheorem{proposition}[theorem]{Proposition}
\theoremstyle{definition} 
\newtheorem{definition}[theorem]{Definition}
\newtheorem*{remark}{Remark}
\newcommand{\comment}[1]{}
\providecommand{\keywords}[1]{\textbf{\textit{Keywords: }} #1}
\DeclareMathOperator{\id}{id}

\DeclareMathOperator{\Id}{Id}
\DeclareMathOperator{\zz}{\mathbb{Z}}

\DeclareMathOperator{\Char}{char}

\title{On some $H$-cleft extensions which are distinguished by their polynomial $H$-identities}
\author{Abel Gomes de Oliveira Jr.\footnote{Department of Mathematics, UFSCar, Brazil, \texttt{abeloliveira@estudante.ufscar.br}} \ and Waldeck Sch\"utzer\footnote{Department of Mathematics, UFSCar, Brazil, \texttt{waldeck@ufscar.br}}}
\date{}
\begin{document}
	\maketitle
	
	\begin{abstract}
		\noindent Let $H_N^q$ be the Taft algebra over a finite commutative ring $R$. When $N$ is a unit in $R$, we show that all $H_N^q$-cleft extensions over $R$ are determined up to $H_N^q$-comodule algebra isomorphism by their polynomial $H_n^q$-identities.
	\end{abstract}
	\keywords{Hopf algebras, comodule algebras, cleft extensions, polynomial identities}

\section*{Introduction}

The main object of study in this paper is the well-known question asking whether the set of polynomials identities distinguishes PI-algebras (associative unital algebras satisfying a nontrivial polynomial identity) up to isomorphism.

When the algebras are defined over an algebraically closed field there are both a range of affirmative results and also of counterexamples. For instance, it follows from the celebrated Amitsur-Levitsky theorem that the standard polynomial of degree $2n$ distinguishes the finite-dimensional central simple associative algebras over an algebraically closed field $k$ up to isomorphism. %

If we assume that the $k$-algebras are ``simple'' (here the meaning of simple depends on the full structure of the algebra), several results came to light in an affirmative way. The interested reader can find such examples for instance in \cite{kush}, \cite{drensky}, \cite{koshlukov}, \cite{aljadeffhaile}, \cite{shestakov} and \cite{bahturin}.

In particular, in \cite{kassel} Kassel extends the notion of graded polynomial identities and considers the so-called polynomial $H$-identities, where $H$ is an arbitrary Hopf algebra. Naturally the problem over algebras became a problem over $H$-comodule algebras, namely algebras which are also right $H$-comodules and the right coaction is compatible with the multiplication. When $H=kG$ is the group Hopf algebra of a group $G$, the $H$-comodule algebras are essentially the same as the $G$-graded associative algebras (cf. \cite[Example 1.6.7]{montgomery}) and the polynomial $H$-identities become the usual $G$-graded polynomial identities. More specifically, Kassel studied certain $H$-comodule algebras which are cleft extensions of the ground field $k$, namely the $H$-Galois objects, for $H$ a (generalized) Taft algebra or the Hopf algebra $E(n)$, and showed that these objects are distinguished by their polynomial $H$-identities. In $\cite{schutzer}$ we extended Kassel's results to the case where $H$ is an arbitrary non-semisimple monomial Hopf algebra.

When the field $k$ is not algebraically closed, the situation can be quite different: the quaternions $\mathbb{H}$ are a central simple algebra of dimension $4$ over $\mathbb{R}$ and $\mathbb{C}\otimes_{\mathbb{R}}\mathbb{H}\cong\mathbb{C}\otimes_{\mathbb{R}}M_2(\mathbb{R})$, hence $\mathbb{H}$ and $M_2(\mathbb{R)}$ have the same set of polynomial identities, but they are obviously not isomorphic as algebras.

In this paper we will investigate $H$-cleft extensions over a (commutative unital) ring $R$. With the intention of being able to make use of a criterion to decide if two cleft extensions are isomorphic we choose $H$ to be a Taft algebra ($H_N^q$) over $R$ and use the $H_N^q$-cleft extensions classification criterion made by Masuoka in \cite{masuoka}.

Under some extra hypothesis only over the ring $R$ we have obtained our main theorem:

\begin{theorem-non} \label{final}
	Fix $N\geqslant 2$ and let $R$ be a finite commutative unital ring such that $N \in R^{\times}$. If $R$ contains a root $q$ of the $N$-th cyclotomic polynomial and the polynomial $H_N^q$-identities of two $H_N^q$-cleft extensions $B$ and $B'$ of $R$ coincide then $B \cong B'$ as $H_N^q$-comodule $R$-algebras.
\end{theorem-non}

Throughout this work, $R$ will denote a commutative ring with identity. Unless it is stated otherwise, expressions like algebra, Hopf algebra, linear and $\otimes$ mean $R$-algebra, Hopf algebra over $R$, $R$-linear and $\otimes_R$, respectively. The group of units of $R$ will be denoted by $R^{\times}$.
	
\section{Preliminaries}

	We assume that the reader is familiar with concepts such as Hopf algebras, $H$-comodules and the Heyneman-Sweedler-type notation for the comultiplication $\Delta\colon H\rightarrow H\otimes H$ and the coaction $\delta \colon M\rightarrow M\otimes H$ for a (right) $H$-comodule $M$.
These concepts can be found in books such as \cite{dascalescu} and \cite{montgomery}.
The theory therein is developed over fields but much of it can
readily generalized for commutative unital rings as can be
seen, for instance, in
\cite[Part II]{chase} and \cite{doi}.
	For convenience, we briefly recall the definition of an $H$-cleft extension and some other related concepts.
		
	Let $H$ be a Hopf algebra and $C$ an algebra.  A right $H$-comodule algebra $B$, with a coaction $\rho \colon B \longrightarrow B\otimes H$, is an \emph{$H$-extension over $C$} if $C=B^{coH}$, where ${B^{coH}=\{b \in B \mid \rho(b)=b \otimes 1\}}$ is the subalgebra of coinvariants of $B$.
		
	An $H$-extension $B$ over $C$ is a \emph{right $H$-Galois extension over $C$} if the map $\beta \colon B\otimes_C B \longrightarrow B\otimes H$ given by $\beta(a\otimes b)=(a\otimes 1)\rho(b)$ is bijective.
		
	An $H$-extension $B$ over $C$ is an \emph{$H$-cleft extension over $C$} if there exists an $H$-comodule map $\varphi \colon H \longrightarrow B$ which is convolution invertible. We can always assume that $\varphi(1)=1$. Such a map $\varphi$ is often called a \emph{section},
and a pair $(B, \varphi)$ is called a \emph{cleft system for $H$ over $C$}.
		
	Cleft and Galois extensions are related concepts. In fact, an $H$-cleft extension $B$ over $C$ can be characterized as a right $H$-Galois extension such that there exists a left $C$-module and right $H$-comodule isomorphism $C \otimes H \cong B$. This is a classical result of Doi and Takeuchi \cite{doi}.
		
	From now on fix an integer $N \geqslant 2$ and assume that the ring $R$ contains a (fixed) root $q$ of the $N$-th cyclotomic polynomial over $\mathbb{Z}$. In particular, $q^N=1$ and given $\zeta_N$ a primitive $N$-th root of 1 in $\mathbb{C}$ there exists a ring map $\mathbb{Z}[\zeta_N] \longrightarrow R$ such that $\zeta_N \mapsto q$.
		
	Denote by $H_N^q$ the Hopf algebra with structure defined as follows: as an algebra, $H_N^q$ is generated by the elements $g, x$ with the relations
	\begin{alignat*}{2}
		g^N=1, &\quad x^N=0, &\quad xg=qgx.
	\end{alignat*}
	As a coalgebra, $H_N^q$ is determined by:
	\begin{alignat*}{2}
		\Delta(g)&=g\otimes g, &&\quad \varepsilon(g)=1, \\
		\Delta(x)&=1\otimes x+x \otimes g, &&\quad \varepsilon(x)=0.
	\end{alignat*}
	Finally, the antipode $S: H_N^q \longrightarrow H_N^q$ is given by
	\begin{alignat*}{2}
		S(g)=g^{N-1}=g^{-1}, &\quad S(x)=-q^{-1}g^{-1}x.
	\end{alignat*}
		
	The Hopf algebra $H_N^q$ is called a Taft (Hopf) algebra.
	
	Given $0\leqslant i\leqslant n\leqslant N$, define the $q$-analog of Newton's Binomial:
	\[
	{n\choose i}_{q} = \frac{(q^n-1)\cdots(q^{n-i+1}-1)}{(q^i-1)\cdots(q-1)}.
	\]
	From \cite[Lemma 2.2]{masuoka}, \cite[Lemma 2.5]{masuoka} and \cite[Prop. IV.2.2]{kasselbook} we obtain the following lemma,
	
	\begin{lemma}\label{qbinomiais}
		Fix $q$ a root of the $N$-th cyclotomic polynomial over $\mathbb{Z}$ and let $z$ and $w$ be two variables such that  $zw=q wz$. Then
		\[
		(z+w)^n = \sum_{i=0}^n {n \choose i}_q z^{i}w^{n-i}.
		\]
		In particular, $(z+w)^N=z^N+w^N$.
	\end{lemma}
	
	By using Bergman's diamond lemma \cite{bergman}, Masuoka \cite[Lemma 2.5]{masuoka} shows that $H_N^q$ is a free $R$-module with basis $\{g^mx^n, \ 0 \leqslant m,n <N\}$.
In that same text, he	
produces a classification of the $H_N^q$-cleft extensions for an algebra $C$. 
Here we wish to consider $C=R$ only, in which case
that classification can be slightly simplified and tailored down
to our needs. We will briefly do so.
	
	\begin{definition} \cite[Def. 2.14]{masuoka}
		Given a triple $\underline{d}=(u,a,b)$, where $u \in R^{\times}$ and $a,b \in R$, define a pair ($\mathcal{B}_{\underline{d}}, \varphi_{\underline{d}}$) as follows: $\mathcal{B}_{\underline{d}}$ is the $R$-algebra  generated by symbols $v_g, v_x$ with relations
		\begin{enumerate}[label=(R\arabic*)]
			\item $v_g^N=u$,
			\item $v_x^N=a$,
			\item $v_xv_g=qv_gv_x+bv_g^2$
		\end{enumerate}
	and $\varphi_{\underline{d}} \colon H_N^q \longrightarrow \mathcal{B}_{\underline{d}}$ is the linear map defined by 
	\[
	\varphi_{\underline{d}}(g^mx^n)=v_g^mv_x^n, \ 0\leqslant m , n <N.
	\]
	\end{definition}

The triple $\underline{d}=(u,a,b)$ is called a \textit{cleft data}.
		
		The following two results intend to ensure that $\mathcal{B}_{\underline{d}}$ is an $H_N^q$-cleft extension of $R$. Moreover, any $H_N^q$-cleft extension of $R$ is isomorphic to $\mathcal{B}_{\underline{d}}$ for some $\underline{d}$.
		
	\begin{proposition}[] \label{cua}\cite[Prop. 2.15]{masuoka}
		\begin{enumerate}[label=\arabic*)]
			\item $\mathcal{B}_{\underline{d}}$ is a free $R$-module with basis $\{v_g^mv_x^n, \ 0 \leqslant m,n <N\}$.
			\item $\mathcal{B}_{\underline{d}}$ has a (right) $H_N^q$-comodule algebra structure determined by
			\begin{align*}
				v_g &\mapsto v_g\otimes g\\
				v_x &\mapsto 1\otimes x+v_x\otimes g.
			\end{align*}
			The subalgebra of coinvariants of $\mathcal{B}_{\underline{d}}$ is $R$.
			\item The map $\varphi_{\underline{d}}\colon H_N^q \longrightarrow \mathcal{B}_{\underline{d}}$ given by 
			\[\varphi_{\underline{d}}(g^mx^n)=v_g^mv_x^n, \ 0 \leqslant m,n <N\]
				 is a section for $\mathcal{B}_{\underline{d}}$.
		\end{enumerate}
		Thus $(\mathcal{B}_{\underline{d}}, \varphi_{\underline{d}})$ is a cleft system for $H_N^q$ over $R$.
	\end{proposition}
	
	\begin{proposition}\label{cleft-caract}\cite[Prop. 2.17]{masuoka}
		Any $H_N^q$-cleft extension over $R$ is isomorphic as $H_N^q$-comodule algebras to $\mathcal{B}_{\underline{d}}$ for some cleft data $\underline{d}$.
	\end{proposition}
		
		Knowing that any $H_N^q$-cleft extension of $R$ is of type $\mathcal{B}_{\underline{d}}$, the next result provides a way to distinguish these cleft extensions.
		
	\begin{proposition}\label{criterio}\cite[Lemma 2.19]{masuoka}
		Let $\mathcal{B}_{\underline{d}}$ and $\mathcal{B}_{\underline{d'}}$ $H_N^q$-cleft extensions of $R$. Then
		\begin{enumerate}
			\item If $F\colon \mathcal{B}_{\underline{d'}} \longrightarrow \mathcal{B}_{\underline{d}}$ is an isomorphism of right $H_N^q$-comodule algebras, then there exists a  pair $(s,t) \in R^{\times}\times R$ such that
			\begin{alignat}{2} \label{isoF}
				F(v_g')=sv_g&\quad \text{ and } &\quad F(v_x')=v_x+tv_g.
			\end{alignat}
		\item If $(s,t) \in R^{\times}\times R$, then the algebra map $F\colon \mathcal{B}_{\underline{d}} \longrightarrow \mathcal{B}_{\underline{d'}}$ given by \cref{isoF} is well defined if and only if
			\begin{equation}
				\begin{aligned}\label{isocond}
					u' &=s^Nu,\\
					a' &= a+t(t+b+qb)(t+b+q^2b)\ldots(t+b+q^{N-1}b)u,\\
					b' &= (b+t-qt)s^{-1}.
				\end{aligned}
			\end{equation}
		In this case, $F$ is a right $H_N^q$-comodule algebra isomorphism.
		\end{enumerate}
	\end{proposition}

	Whenever $1-q \in R^{\times}$, we can further reduce the cleft data from 3 to 2 parameters and simplify \Cref{criterio}.

	\begin{corollary} \label{cor b=0}
		Let $\mathcal{B}_{(u,a,b)}$ be an $H_N^q$-cleft extension of $R$. If $1-q \in R^{\times}$ then there exists $a' \in R$ such that $\mathcal{B}_{(u,a,b)} \cong \mathcal{B}_{(u,a',0)}$.
	\end{corollary}
	\begin{proof}
		By \Cref{criterio} it is enough to find a pair $(s,t) \in R^{\times}\times R$ in order to satisfy the equations \eqref{isocond}. Taking $s=1$ and $t=-\frac{b}{1-q}$, the first and the third equations are obviously satisfied and the value of $a'$ is then obtained from the second one.
	\end{proof}
	
	In view of this corollary, from now on we denote $\mathcal{B}_{(u,a,0)}$ simply by $\mathcal{B}_{(u,a)}$. \Cref{criterio} and \Cref{cor b=0} further provide a means to distinguish such 
cleft extensions by straightforward direct computation:
	
	\begin{proposition}\label{novocriterio}
			Let $\mathcal{B}_{(u,a)}$ and $\mathcal{B}_{(u',a')}$ be $H$-cleft extensions of $R$. Then $\mathcal{B}_{(u,a)}\cong \mathcal{B}_{(u',a')}$ if and only if $a'=a$ and there exists $s \in R^{\times}$ such that $u'=s^Nu$.
	\end{proposition}

\section{Polynomial identities}
	
	Drawing from Kassel \cite{kassel}, let $H$ be a Hopf algebra over $R$ and assume that $H$ is a free $R$-module. For each ${i=1,2,\ldots}$ let $Z_i^H$ be a copy $H$ as a free $R$-module and denote by $Z_i^h \ (h \in H)$ the elements of $Z_i^H$ (called $Z$-symbols). Define $Z_H=\bigoplus_{i\geqslant 1} Z_i^H$ and take the tensor algebra
	\[
	T=T(Z_H)=T\left( \displaystyle \bigoplus_{i\geqslant 1} Z_i^H \right).
	\]
	$T$ is a (right) $H$-comodule algebra with coaction given by
	\[
	\delta(Z_i^h)=Z_i^{h_1}\otimes h_2.
	\]
It is useful to recall that we are using Sweedler's summation
notation for the coproduct, namely that $h_1\otimes h_2=\Delta(h)=\sum_{i}h_{1i}\otimes h_{2i}$.

	Given a linear basis $\{h_1, h_2, \ldots\}$ for $H$, it is easy to see that the tensor algebra $T$ is isomorphic to the free associative unital algebra given by the indeterminates $\bigcup_{i\geq 1}\{Z_i^{h_j}\mid j=1,2,\ldots\}$. In view of this
	remark it should be clear that the following definition generalizes both the ordinary (case $H=R$) and the $G$-graded (case $H=RG$) polynomial identities:
	\begin{definition}\cite[Def. 2.1]{kassel}
		An element $P \in T$ is a polynomial $H$-identity for a right $H$-comodule algebra $B$ if $f(P)=0$ for all right $H$-comodule algebra maps $f \colon T \longrightarrow B$.
	\end{definition}
	
	We denote the set of the polynomial $H$-identities for a right  $H$-comodule algebra $B$ by $\Id_H(B)$. Then
	\[
	\Id_H(B)=\displaystyle\bigcap_{f} \ker f,
	\]
	where $f$ runs over all right $H$-comodule algebra maps $T \longrightarrow B.$
	
	For a Taft algebra $H_N^q$ fix a linear $R$-basis $\{g^mx^n, 0\leqslant m,n < N\}$ as before and denote $E=Z_1^1$, $G=Z_1^g$ and $X=Z_1^x$. Then
	\begin{alignat*}{2}
		\delta(E)&=E\otimes 1\\
		\delta(G)&=G\otimes g\\
		\delta(X)&=E\otimes x +X \otimes g.
	\end{alignat*}
	
	\begin{proposition} \label{cara dos hom}
		Let $H_N^q$ be the Taft algebra, $T$ the tensor algebra over $H_N^q$ and $\mathcal{B}_{(u,a)}$ an $H_N^q$-cleft extension of $R$. For each right $H_N^q$-comodule algebra map ${f \colon T \longrightarrow \mathcal{B}_{(u,a)}}$ there exist unique $\lambda, \mu, \xi \in R$ such that
		\begin{enumerate}
			\item[i)] $f (E)=\lambda$,
			\item[ii)] $f (G)=\mu v_g$,
			\item[iii)] $f (X)=\lambda v_x+\xi v_g$.
		\end{enumerate}
	\end{proposition}
	\begin{proof}
		Fix a linear $R$-basis $\{v_g^mv_x^n\,|\, 0\leqslant m,n < N\}$  for $\mathcal{B}_{(u,a)}$ as in \Cref{cua}, in which case the coaction $\rho \colon \mathcal{B}_{(u,a)}\longrightarrow \mathcal{B}_{(u,a)} \otimes H_N^q$ is given by
		\begin{alignat*}{2}
			\rho(v_g)&=v_g\otimes g\\
			\rho(v_x)&=1\otimes x +v_x \otimes g.
		\end{alignat*}
		 Then for each $P \in T$, there exist unique coeficients $\alpha_{m,n}^{P}\in R$ ($0\leqslant m,n < N$) such that
		\[
		f(P)=\displaystyle\sum_{m,n=0}^{N-1}\alpha_{m,n}^Pv_g^mv_x^n.
		\]
Since $f$ is an $H_N^q$-comodule map, namely ${\rho \circ f=(f \otimes \id) \circ \delta}$, we can compute the values
of the coefficients $\alpha_{m,n}^P$ when $P$ is one of $E$, $G$ and $X$ as follows. First, from \Cref{qbinomiais}, we have that
		\begin{alignat*}{2}
			\rho \circ f(P)&=\displaystyle\sum_{m,n=0}^{N-1}\alpha_{m,n}^P\rho(v_g)^m\rho(v_x)^n\\
			&=\displaystyle\sum_{m,n=0}^{N-1}\alpha_{m,n}^P(v_g\otimes g)^m(1\otimes x+v_x\otimes g)^n\\
			&=\displaystyle\sum_{m,n=0}^{N-1}\alpha_{m,n}^P\displaystyle\sum_{l=0}^{n}\binom{n}{l}_qv_g^mv_x^{n-l}\otimes g^{m+n-l}x^l,
		\end{alignat*}
which is a tensor expression in terms of an $R$-basis of
$\mathcal{B}_{(u,a)}\otimes H_N^q$.
		
Now, since $(f\otimes \id)\circ\delta(P)=\rho\circ f(P)$,
we have:	\begin{enumerate}
			\item[i)] For $P=E$,
			\begin{alignat*}{2}
				(f \otimes \id) \circ 	\delta(E)&=f(E)\otimes 1\\
				&=\left(\displaystyle\sum_{m,n=0}^{N-1}\alpha_{m,n}^Ev_g^mv_x^n\right)\otimes 1\\
				&=\left( \alpha_{0,0}^E+ 	\displaystyle\sum_{m=1}^{N-1}\alpha_{m,N-m}^Ev_g^mv_x^{N-m}+  \displaystyle\sum_{\substack{m,n=0\\ m+n\neq 0\\ m+n\neq N}}^{N-1} \alpha_{m,n}^E v_g^mv_x^n \right) \otimes 1
			\end{alignat*}
			
			and
			\begin{alignat*}{2}
				\rho \circ f(E)&=\left( \alpha_{0,0}^E+ 	\displaystyle\sum_{m=1}^{N-1}\alpha_{m,N-m}^Ev_g^mv_x^{N-m}\right) \otimes 1+  \displaystyle\sum_{m=1}^{N-1} \alpha_{m,N-m}^Ev_g^m \otimes g^m+\\
				&+\displaystyle\sum_{m=1}^{N-1} \alpha_{m,N-m}^E \displaystyle\sum_{l=1}^{N-m-1}   \binom{N-m}{l}_qv_g^mv_x^{N-m-l} \otimes g^{N-l}x^l+\\
				&+\displaystyle\sum_{\substack{m,n=0\\ m+n\neq 0\\ m+n\neq N}}^{N-1} \alpha_{m,n}^E \displaystyle\sum_{l=0}^{n}   \binom{n}{l}_qv_g^mv_x^{n-l} \otimes g^{m+n-l}x^l.
			\end{alignat*}
			Comparing coefficients in these equations we get $\alpha_{m,n}^E=0$ whenever ${(m,n) \neq (0,0)}$ and $\alpha_{0,0}^E$ is the only remaining quantity, so we can take $\lambda:=\alpha_{0,0}^E$ and obtain $f(E)=\lambda$.
			\item[ii)] For $P=G$,
			\begin{alignat*}{2}
				(f \otimes \id) \circ \delta(G)&=f(G)\otimes g\\ &=\left(\displaystyle\sum_{m,n=0}^{N-1}\alpha_{m,n}^Gv_g^mv_x^n\right)\otimes g\\
				&=\left( \alpha_{1,0}^Gv_g+ \alpha_{0,1}^Gv_x+	\displaystyle\sum_{m=2}^{N-1}\alpha_{m,N+1-m}^Gv_g^mv_x^{N+1-m}\right)\otimes g+ \\ &+\displaystyle\sum_{\substack{m,n=0\\ m+n\neq 1\\ m+n\neq N+1}}^{N-1} \alpha_{m,n}^G v_g^mv_x^n \otimes g
			\end{alignat*}
			and
			\begin{alignat*}{2}
				\rho \circ f(G)&=\left( \alpha_{1,0}^Gv_g+ \alpha_{0,1}^Gv_x+	\displaystyle\sum_{m=2}^{N-1}\alpha_{m,N+1-m}^Gv_g^mv_x^{N+1-m} \right) \otimes g+\\
				&+\alpha_{0,1}^G\otimes x+\displaystyle\sum_{m=2}^{N-1} \alpha_{m,N+1-m}^G v_g^m \otimes g^mx^{N+1-m}+\\
				&+\displaystyle\sum_{m=2}^{N-1} \alpha_{m,N+1-m}^G \displaystyle\sum_{l=1}^{N+1-m}   \binom{N+1-m}{l}_qv_g^mv_x^{N+1-m-l} \otimes g^{N+1-l}x^l+\\
				&+\displaystyle\sum_{\substack{m,n=0\\ 1\neq m+n\neq N+1}}^{N-1} \alpha_{m,n}^G \displaystyle\sum_{l=0}^{n}   \binom{n}{l}_qv_g^mv_x^{n-l} \otimes g^{m+n-l}x^l.
			\end{alignat*}
			Again, comparing coefficients in these equations, we get that $\alpha_{m,n}^G=0$ whenever ${(m,n) \neq (1,0)}$ and that only $\alpha_{1,0}^G$ remains undetermined. Taking $\mu:=\alpha_{1,0}^G$, we get $f(G)=\mu v_g$.
			\item[iii)] For $P=X$, since $\alpha_{m,n}^E=0$ whenever $(m,n)\neq (0,0)$,
			\begin{alignat*}{2}
				(f \otimes \id) \circ \delta(X)&= f(E)\otimes x +f(X)\otimes g\\
				&=\alpha_{0,0}^E\otimes x+\displaystyle\sum_{m,n=0}^{N-1}\alpha_{m,n}^Xv_g^mv_x^n\otimes g\\
				&=\alpha_{0,0}^E\otimes x +  \alpha_{1,0}^Xv_g \otimes g+ \alpha_{0,1}^Xv_x \otimes g+ \\ &	+\displaystyle\sum_{m=2}^{N-1}\alpha_{m,N+1-m}^Xv_g^mv_x^{N+1-m} \otimes g +\displaystyle\sum_{\substack{m,n=0\\ m+n\neq 1\\ m+n\neq N+1}}^{N-1} \alpha_{m,n}^X v_g^mv_x^n \otimes g
			\end{alignat*}
			and
			\begin{alignat*}{2}
				\rho \circ f(X)&=\left( \alpha_{1,0}^Xv_g+ \alpha_{0,1}^Xv_x+	\displaystyle\sum_{m=2}^{N-1}\alpha_{m,N+1-m}^Xv_g^mv_x^{N+1-m}+ \right) \otimes g+\\
				&+\left(\alpha_{0,1}^X+\displaystyle\sum_{m=2}^{N-1} \alpha_{m,N+1-m}^X \binom{N+1-m}{1}_q v_g^mv_x^{N-m} \right) \otimes x+\\
				&+\displaystyle\sum_{m=2}^{N-1} \alpha_{m,N+1-m}^X v_g^m \otimes g^mx^{N+1-m}+\\
				&+\displaystyle\sum_{m=2}^{N-1} \alpha_{m,N+1-m}^X \displaystyle\sum_{l=2}^{N+1-m-l}\binom{N+1-m}{l}_qv_g^mv_x^{N+1-m-l} \otimes g^{N+1-l}x^l+\\
				&+\displaystyle\sum_{\substack{m,n=0\\ 1\neq m+n\neq N+1}}^{N-1} \alpha_{m,n}^X \displaystyle\sum_{l=0}^{n} \binom{n}{l}_qv_g^mv_x^{n-l} \otimes g^{m+n-l}x^l.
			\end{alignat*}
			This time, comparing coefficients gives us that $\alpha_{1,0}^X$ is free, ${\alpha_{0,1}^X=\alpha_{0,0}^E}=\lambda$ and $\alpha_{m,n}^X=0$ whenever $(m,n) \neq (1,0)$ and $(m,n)\neq (0,1)$. Taking ${\xi: =\alpha_{1,0}^X}$, we get $f(X)=\lambda v_x+\xi v_g$.
		\end{enumerate}
	\end{proof}
	
	Obviously, the previous proposition remains true in case the $E$, $G$ and $X$ are redefined as $E_i=Z_i^1$, $G_i=Z_i^g$ and $X_i=Z_i^x$, for all $i\geq 1$.

\section{The Isomorphism Question}
	
	Back to our main question: For the Taft algebra $H_N^q$, given $B$ and $B'$ $H_N^q$-cleft extensions over $R$, when $\Id_{H_N^q}(B)=\Id_{H_N^q}(B')$ implies $B \cong B'$ as $H_N^q$-comodule algebras?
	
	In view of \Cref{cleft-caract}, \Cref{cor b=0} and \Cref{novocriterio}, we can rearrange this question as follows: When $\Id_{H_N^q}(\mathcal{B}_{(u,a)})=\Id_{H_N^q}(\mathcal{B}_{(u',a')})$ implies that $a'=a$ and there exists $s \in R^{\times} $ such that $u'=s^Nu$? We analyze this question dividing it naturally in two parts. For each part we dedicate one subsection. 
	
	\subsection{The $a'=a$ question}
		
	The purpose of this subsection is to prove that in order to obtain that $\Id_{H_N^q}(\mathcal{B}_{(u,a)})=\Id_{H_N^q}(\mathcal{B}_{(u',a')})$ implies $a'=a$ is enough to ask that $1-q \in R^{\times}$. Note that this condition is the same of \Cref{cor b=0}. In other words, there is no loss in asking that.
	
	In order to get that we first select select a suitable polynomial $H_N^q$-identity:
	
		\begin{lemma} \label{pa}
			Let $\mathcal{B}_{(u,a)}$ be an $H_N^q$-cleft extension over $R$.	Then 
			\[
			\mathcal{P}_a=(XG-q GX)^N-(1-q)^NG^NX^N+(1-q)^NaE^NG^N
			\]
			is a polynomial $H_N^q$-identity for $\mathcal{B}_{(u,a)}$.
		\end{lemma}
	\begin{proof}
		Let $f \colon T \longrightarrow \mathcal{B}_{(u,a)}$ an $H_N^q$-comodule algebra map. By \Cref{cara dos hom} and \Cref{qbinomiais} we have
		\begin{align*}
			f((XG-q GX)^N)&=(1-q)^{N}\mu^N\xi^Nu^{2}\\
			f((1-q)^NG^NX^N)&=(1-q)^N\mu^N(\xi^Nu^2+\lambda^Nua)\\
			f((1-q)^NaE^NG^N)&=(1-q)^Na\lambda^N\mu^Nu
		\end{align*}
		Therefore $f(\mathcal{P}_a)=0$.
	\end{proof}
		
		The polynomial $\mathcal{P}_a$ has been used by Kassel in \cite{kassel} to solve the same type of question when $R$ is an algebraically closed field. To this end, Kassel defines a map (called universal map) with the following property: given a polynomial $P$, $P$ is a polynomial $H$-identity if and only if $P$ is in the kernel of this universal map. However such a map doesn't work so well when $R$ is not an algebraically closed field. In this case, the universal map still can be defined and every polynomial in its kernel is a polynomial $H$-identity. However it is possible to exists polynomial $H$-identities that are not in the kernel of the universal map. In our case what happens is that $\mathcal{P}_a$ is  actually in the kernel of that universal map, but  a further polynomial $H$-identity  \eqref{G-polynomial} will not. Therefore  we opted for use \Cref{cara dos hom} in both polynomials.
		
	\begin{proposition} \label{a'a}
		Let $\mathcal{B}_{(u,a)}$ and $\mathcal{B}_{(u',a')}$ be two $H_N^q$-cleft extensions over $R$. If $1-q \in R^{\times}$ and
		\[
		\Id_{H_N^q}(\mathcal{B}_{(u,a)})=\Id_{H_N^q}(\mathcal{B}_{(u',a')}),
		\]
		then $a'=a$.
	\end{proposition}	
	\begin{proof}
		By \Cref{pa}, $\mathcal{P}_a$ is a polynomial $H_N^q$-identity for $\mathcal{B}_{(u,a)}$ and $\mathcal{P}_{a'}$ is a polynomial $H_N^q$-identity for $\mathcal{B}_{(u',a')}$. As the set of a polynomial $H_N^q$-identities coincide,  $\mathcal{P}_a-\mathcal{P}_{a'}=(1-q)^N(a-a')E^NG^N$ is a polynomial $H_N^q$-identity for $\mathcal{B}_{(u,a)}$.
		
		Then for any $H_N^q$-comodule algebra map $f \colon T \longrightarrow \mathcal{B}_{(u,a)}$, we have $f(\mathcal{P-P'})=0$. 
		In particular, the section $\varphi \colon H \longrightarrow \mathcal{B}_{(u,a)}$ given by $\gamma (g^mx^n)=v_g^mv_x^n$ (an $H_N^q$-comodule linear map) defines an $H_N^q$-comodule algebra map $\Gamma \colon T \longrightarrow \mathcal{B}_{(u,a)}$ such that $\Gamma(E)=\gamma(1)=1$ and $\Gamma(G)=\gamma(g)=v_g$.

		Therefore $0=\Gamma  (\mathcal{P-P'})=(1-q)^N(a-a')u$. 
		Since $	(1-q)^N$ and $u$ are both in $R^{\times}$, we obtain $a'=a$.
	\end{proof}	
		
	\subsection{The $u'=s^Nu$ question}
	After ensuring that ${\Id_{H_N^q}(\mathcal{B}_{(u,a)})=\Id_{H_N^q}(\mathcal{B}_{(u',a')})}$ implies $a'=a$ with no extra conditions (since $1-q \in R^{\times}$ was already a condition), now, we will give sufficient conditions to also obtain $u'=s^Nu$. 
	
	We will make use of two extra hypotheses: $R$ finite and $N \in R^{\times}$.
	
	Remember some general properties about finite rings:
	
	\begin{remark} Let $R$ be a finite commutative unital ring.  Then
	\begin{enumerate}[label=\arabic*)]
		\item Local rings, indecomposable rings and rings with no proper idempotents are equivalent concepts.
		\item If $R$ has any of the equivalent properties in 1) then every element of $R$ is either a unit or nilpotent. That is, for any $ x \in R$, either $x^{\alpha}=1$ or $x^{\beta}=0$, where $\alpha$ is the exponent of $R^{\times}$ and $\beta$ is the index of nilpotency of $R$.
		\item  $R$ is Artinian. Therefore, by \cite[Prop. 22.1 and Prop. 2.22]{lam}, the unity $1_R$ decomposes into a sum of primitive pairwise orthogonal idempotents ${1_R=e_1+ \ldots+e_n}$ uniquely. In particular, for each $r \in R$, the decomposition ${r=re_1+\ldots +re_n}$ is unique. This induces a (block) decomposition $R=R_1\times \ldots \times R_n$ of local rings, where $R_i:=Re_i$ and $e_i$ is the identity element of $R_i$ for any $i=1, \ldots, n$.
		\item If $(N, \Char R)=1$ then $N \in R^{\times}$.
	\end{enumerate}
\end{remark}
	
	From now on, whenever $R$ is a finite ring,  fix $\alpha$ to be the exponent of $R^{\times}$ and $\beta$ to be the nilpotence index of $R$.
	
	If  $N \in R^{\times}$ next lemma bring some extra properties for a finite ring.
	
	\begin{lemma} \label{lemma}
		Let $N\geqslant 2$ and suppose there exists $q$ a $N$-th root of the cyclotomic polynomial in $R$. If $R$ is finite  and $N \in R^{\times}$. Then:
		\begin{enumerate}
			\item $(N, \Char R)=1$,
			\item $N=o(q)$. In particular $N \mid \alpha$,
			\item $1-q \in R^{\times}$.
		\end{enumerate}
	\end{lemma}
	\begin{proof}$ $
		\begin{enumerate}
			\item Since $R$ is finite, $\Char R>0$. Let $(N, \Char R)=d>0$. Therefore there exists $x,y>0$ such that $N=dx$ and $\Char R=dy$. $N \in R^{\times}$ implies $d \in R^{\times}$. On the other hand, $ay=0$ in $R$. Since $a \in R^{\times}$, $y=0$ in $R$. Thus $\Char R \mid y$. Since $\Char R\neq 0$, $y=\Char R >0$ and therefore $a=1$.
			\item Let $m:=o(q)$. Given $\mathfrak{m}$ a maximal ideal of $R$, define $\mathbb{F}=R/\mathfrak{m}$ (the field with elements $\overline{r}:=r+\mathfrak{m}, r \in R$). Take the polynomial $f(x)=x^N-1$ in $\mathbb{F}[x]$.  Since $q^N=1$, we obtain $\overline{q}^N=1$ and therefore $f(\overline{q})=0$. By \cite[IV, \S 1, Prop. 1.11]{lang}, $\overline{q}$ is a simple root of $f$ if and only if $f'(q)=N\overline{q}^{N-1} \neq 0$. That is the case since $N \in R^{\times}$.
			
			On the other hand, using cyclotomic polynomials and because $m \mid N$,
			\begin{align*}
				f(\overline{q})=\overline{q}^N-1&=\prod_{d \mid N} \Phi_d(\overline{q})\\
				&= \left( \prod_{d \mid m} \Phi_d(\overline{q}) \right) \left(\prod_{\substack{d \mid N\\ d \nmid m \\d \neq N}} \Phi_d(\overline{q}) \right) \Phi_N(\overline{q})\\
				&= (\overline{q}^m-1) g(\overline{q}) \Phi_N(\overline{q})
			\end{align*}
			Observe that $\overline{q}^m-1=0$ and $\Phi_N(\overline{q})=0$. Since $q$ is a simple root of $f$, necessarily $o(q)=m=N$. In particular, since $R$ is finite, Lagrange's theorem guarantee $N \mid \alpha$.
			\item Using the block decomposition of $R$ by a family of orthogonal primitive idempotents $\{e_i\}$, let $1-q=\sum_i (1-q)e_i$. In order to obtain $1-q \in R^{\times}$ it is enough to prove that $(1-q)e_i \in R^{\times}_i$ for each $i$. Since $N, e_i \in R^{\times}_i$ and $R_i$ indecomposable, we have by \cite[Cor. 2.4]{janusz} that the polynomial $f(x)=x^N-e_i$ is separable. Since $1$ and $q$ are distinct roots of $f(x)$ ($q \neq 1$ because $o(q)=N\geqslant 2$), one obtain by \cite[Lemma 2.1]{janusz} that $1-q \in R^{\times}$.
		\end{enumerate}
	\end{proof}
		
	Take $G=Z_1^g$ as defined in the previous section and whenever $N \mid \alpha$,  set $k$ to be the integer such that $\alpha=kN$.
	
	As we did in the previous subsection, we are looking for a suitable polynomial  $H_N^q$-identity to help us solve this subsection question.
	\begin{proposition}\label{Qpol}
		Let $\mathcal{B}_{(u,a)}$ be an $H_N^q$-cleft extension of $R$. If $R$ is finite and $N \in R^{\times}$ then
		\begin{equation} \label{G-polynomial}
			\mathcal{Q}_u=(u^k-G^{\alpha})G^{\beta}
		\end{equation}
		is a polynomial $H_N^q$-identity for $\mathcal{B}_{(u,a)}$.
	\end{proposition}
	\begin{proof}
		Let $f \colon T \longrightarrow \mathcal{B}_{(u,a)}$ an $H_N^q$-comodule algebra map. By \Cref{cara dos hom} we have that
		\[
			f((u^k-G^{\alpha})G^{\beta})=u^k(1-\mu^{\alpha})\mu^{\beta}v_g^{\beta}.
		\]
		Observe that $f(\mathcal{Q}_u)=0$ if and only if $(1-\mu^{\alpha})\mu^{\beta}=0$. Since $R$ is finite, there exists a family of idempotents $\{e_i\}_{i=1}^n$ inducing a block decomposition $R=R_1 \times \ldots \times R_n$. Then
		\begin{align*}
			(1-\mu^{\alpha})\mu^{\beta}&=\sum_{i=1}^n (1-\mu^{\alpha})\mu^{\beta}e_i\\
			&=\sum_{i=1}^n \left(1e_i-(\mu e_i)^{\alpha}\right)(\mu e_i)^{\beta}\\
			&=0
		\end{align*}
		
		The last step follows by using that for all $\mu \in R$ and $i=1, \ldots, n$, $\mu e_i$ is a unit or a niplotent element in $R_i$. If $\mu e_i$ is a unit then $(\mu e_i)^{\alpha}=1e_i$. If $(\mu e_i) $ is nilpotent then $(\mu e_i)^{\beta}=0$.
		
		Therefore $f(\mathcal{Q}_u)=0$.
	\end{proof}	

	Unlike the $\mathcal{P}_a$ polynomial, $\mathcal{Q}_u$ does not solve the $u'=s^Nu$ question directly, instead:

	\begin{proposition} \label{uk}
		Let $\mathcal{B}_{(u,a)}$ and $\mathcal{B}_{(u',a')}$ be two $H_N^q$-cleft extensions of $R$. If $\Id_{H_N^q}(\mathcal{B}_{(u,a)})=\Id_{H_N^q}(\mathcal{B}_{(u',a')})$ then $(u')^k=u^k$, where $k$ is such that $\alpha=kN$.
	\end{proposition}
	\begin{proof}
		Consider the polynomial $H_N^q$-identity $\mathcal{Q}_u$ from \cref{G-polynomial} of $\mathcal{B}_{(u,a)}$. Since the polynomial $H_N^q$-identities coincide, $\mathcal{Q}_u \in \Id_{H_N^q}(\mathcal{B}_{(u',a')})$. Then for any  $H_N^q$-comodule algebra map $f \colon T \longrightarrow \mathcal{B}_{(u,a)}$
		\[
		0=f(\mathcal{Q}_u)=(u^k-(u')^k\mu^{\alpha})\mu^{\beta}v_g^{\beta}
		\]
		In particular, for $\mu=1$, we have that $u^k=(u')^k$.
	\end{proof}	

	However, since we can use the fundamental theorem of finitely generated abelian groups, next two results show that $u^k=(u')^k$ it is enough to obtain $s \in R^{\times}$ such that $u'=s^Nu$.

	\begin{lemma}\label{FTFAG}
		Let $G$ be a finite abelian group and $\alpha$ the exponent of $G$. Then $G$ is generated by a subset of $\{g \in G \mid o(g)=\alpha\}$.
	\end{lemma}
	\begin{proof}
		By the fundamental theorem of finitely generated abelian groups \cite[Theorem 5.3]{dummit}, $G\cong \zz_{n_1}\oplus \ldots \oplus \zz_{n_k}$ with $n_1 \mid  \ldots \mid n_k$. Since $\zz_{n_i}$ is cyclic for any $i$, the exponent and the order of each factor group coincide. Therefore $n_k=\alpha$. Take $S= \{(1, \ldots, 1), (0,1,\ldots, 1), (0,0,1,\ldots, 1)), \ldots, (0,\ldots,0, 1)\}$. It is not hard to see that $S$ generates $G$ and $o(s)=\alpha$ for all $s \in S$.
	\end{proof}

	\begin{proposition} \label{lema cyclic}
		Let $N\geqslant 2$ and suppose there exists $q$ a $N$-th root of the cyclotomic polynomial in $R$. If $R$ is finite  and $N \in R^{\times}$. Then for each $c \in R^{\times}$, the polynomial $f(x)=x^N-c$ has a root in $R$ if and only if $c^k=1$.
	\end{proposition}
	\begin{proof}
		If exists $t \in R^{\times}$ such that $t^N=c$ then $c^k=t^{Nk}=t^{\alpha}=1$.
		On the other hand, since by \Cref{FTFAG}, a subset of $\{y \in R^{\times} \mid o(y)=\alpha\}$ generates $R^{\times}$, then there exists $y \in R^{\times}$ and $b>0$ such that $c=y^b$, with $o(y)=\alpha$. Then $y^{bk}=1$. Therefore exists $m>0$ such that $bk=\alpha m=Nkm$. Then $b=Nm$. So that we can write $c=y^b=(y^m)^N$. Taking $t=y^{m}$ we have $f(t)=0$. 
		
	\end{proof}	

	All this together solves the question of this subsection:
	\begin{proposition} \label{usnu}
		Let $\mathcal{B}_{(u,a)}$ and $\mathcal{B}_{(u',a')}$ be two $H_N^q$-cleft extensions of $R$. If $R$ is finite, $N \in R^{\times}$ and $\Id_{H_N^q}(\mathcal{B}_{(u,a)})=\Id_{H_N^q}(\mathcal{B}_{(u',a')})$ then there exists $s \in R^{\times}$ such that $u'=s^Nu$. 
	\end{proposition}
	\begin{proof}
		By \Cref{uk}, $(u')^k=u^k$. Taking $c=u'u^{-1}$ in \Cref{lema cyclic} we obtain $u'=s^Nu$.
	\end{proof}	

	Thereby, our main theorem (in the introduction) follows directly from \Cref{a'a} and \Cref{usnu} in view of \Cref{lemma} and \Cref{criterio}.

\section*{Acknowldegments}
This study was financed in part by the Coordenação de
Aperfeiçoamento de Pessoal de Nível Superior - Brasil (CAPES) - Finance Code 001

\bibliographystyle{plain}
\addcontentsline{toc}{chapter}{Bibliografia}
\bibliography{biblio}

\begin{thebibliography}{10}

\bibitem{aljadeffhaile}
Eli Aljadeff and Darrell Haile.
\newblock Simple {$G$}-graded algebras and their polynomial identities.
\newblock {\em Trans. Amer. Math. Soc.}, 366(4):1749--1771, 2014.

\bibitem{bahturin}
Yuri Bahturin and Felipe Yasumura.
\newblock Graded polynomial identities as identities of universal algebras.
\newblock {\em Linear Algebra and its Applications}, 562:1 -- 14, 2019.

\bibitem{bergman}
George~M. Bergman.
\newblock The diamond lemma for ring theory.
\newblock {\em Adv. in Math.}, 29(2):178--218, 1978.

\bibitem{chase}
Stephen~U. Chase and Moss~E. Sweedler.
\newblock {\em Hopf algebras and {G}alois theory}.
\newblock Lecture Notes in Mathematics, Vol. 97. Springer-Verlag, Berlin-New
  York, 1969.

\bibitem{doi}
Yukio Doi and Mitsuhiro Takeuchi.
\newblock Cleft comodule algebras for a bialgebra.
\newblock {\em Comm. Algebra}, 14(5):801--817, 1986.

\bibitem{drensky}
V.~S. Drensky and M.~L. Racine.
\newblock Distinguishing simple {J}ordan algebras by means of polynomial
  identities.
\newblock {\em Comm. Algebra}, 20(2):309--327, 1992.

\bibitem{dascalescu}
Sorin D\u{a}sc\u{a}lescu, Constantin N\u{a}st\u{a}sescu, and \c{S}erban Raianu.
\newblock {\em Hopf algebras}, volume 235 of {\em Monographs and Textbooks in
  Pure and Applied Mathematics}.
\newblock Marcel Dekker, Inc., New York, 2001.
\newblock An introduction.

\bibitem{dummit}
David~S. Dummit and Richard~M. Foote.
\newblock {\em Abstract algebra}.
\newblock John Wiley \& Sons, Inc., Hoboken, NJ, third edition, 2004.

\bibitem{janusz}
G.~J. Janusz.
\newblock Separable algebras over commutative rings.
\newblock {\em Trans. Amer. Math. Soc.}, 122:461--479, 1966.

\bibitem{kasselbook}
Christian Kassel.
\newblock {\em Quantum groups}, volume 155 of {\em Graduate Texts in
  Mathematics}.
\newblock Springer-Verlag, New York, 1995.

\bibitem{kassel}
Christian Kassel.
\newblock Examples of polynomial identities distinguishing the {G}alois objects
  over finite-dimensional {H}opf algebras.
\newblock {\em Ann. Math. Blaise Pascal}, 20(2):175--191, 2013.

\bibitem{koshlukov}
Plamen Koshlukov and Mikhail Zaicev.
\newblock Identities and isomorphisms of graded simple algebras.
\newblock {\em Linear Algebra Appl.}, 432(12):3141--3148, 2010.

\bibitem{kush}
A.~Kh. Kushkule\u{\i} and Yu.~P. Razmyslov.
\newblock Varieties generated by irreducible representations of {L}ie algebras.
\newblock {\em Vestnik Moskov. Univ. Ser. I Mat. Mekh.}, (5):4--7, 1983.

\bibitem{lam}
T.~Y. Lam.
\newblock {\em A first course in noncommutative rings}, volume 131 of {\em
  Graduate Texts in Mathematics}.
\newblock Springer-Verlag, New York, 1991.

\bibitem{lang}
Serge Lang.
\newblock {\em Algebra}, volume 211 of {\em Graduate Texts in Mathematics}.
\newblock Springer-Verlag, New York, third edition, 2002.

\bibitem{masuoka}
Akira Masuoka.
\newblock Cleft extensions for a {H}opf algebra generated by a nearly primitive
  element.
\newblock {\em Comm. Algebra}, 22(11):4537--4559, 1994.

\bibitem{montgomery}
Susan Montgomery.
\newblock {\em Hopf algebras and their actions on rings}, volume~82 of {\em
  CBMS Regional Conference Series in Mathematics}.
\newblock Published for the Conference Board of the Mathematical Sciences,
  Washington, DC; by the American Mathematical Society, Providence, RI, 1993.

\bibitem{schutzer}
Waldeck Sch\"{u}tzer and Abel~Gomes de~Oliveira.
\newblock On some {$H$}-{G}alois objects and their polynomial {$H$}-identities.
\newblock {\em Arch. Math. (Basel)}, 116(1):7--18, 2021.

\bibitem{shestakov}
Ivan Shestakov and Mikhail Zaicev.
\newblock Polynomial identities of finite dimensional simple algebras.
\newblock {\em Comm. Algebra}, 39(3):929--932, 2011.

\end{thebibliography}
\clearpage \thispagestyle{empty}

\end{document}